\documentclass[runningheads]{llncs}
\usepackage{graphicx}

\usepackage[T1]{fontenc}
\usepackage[utf8]{inputenc}

\usepackage[hidelinks]{hyperref}

\usepackage{amsmath}
\usepackage{mathtools}
\usepackage{csquotes}
\usepackage{enumitem}

\graphicspath{{figures/}}
\newcommand{\fig}[1]{Fig.~\ref{#1}}

\newcommand{\ceil}[1]{\left\lceil {#1} \right\rceil}
\newcommand{\abs}[1]{\left\lvert {#1} \right\rvert}
\newcommand{\cF}{\mathcal{F}}
\newcommand{\mmm}{\mathrm{mm}}

\renewcommand{\orcidID}[1]{\href{https://orcid.org/#1}{\includegraphics[scale=.03]{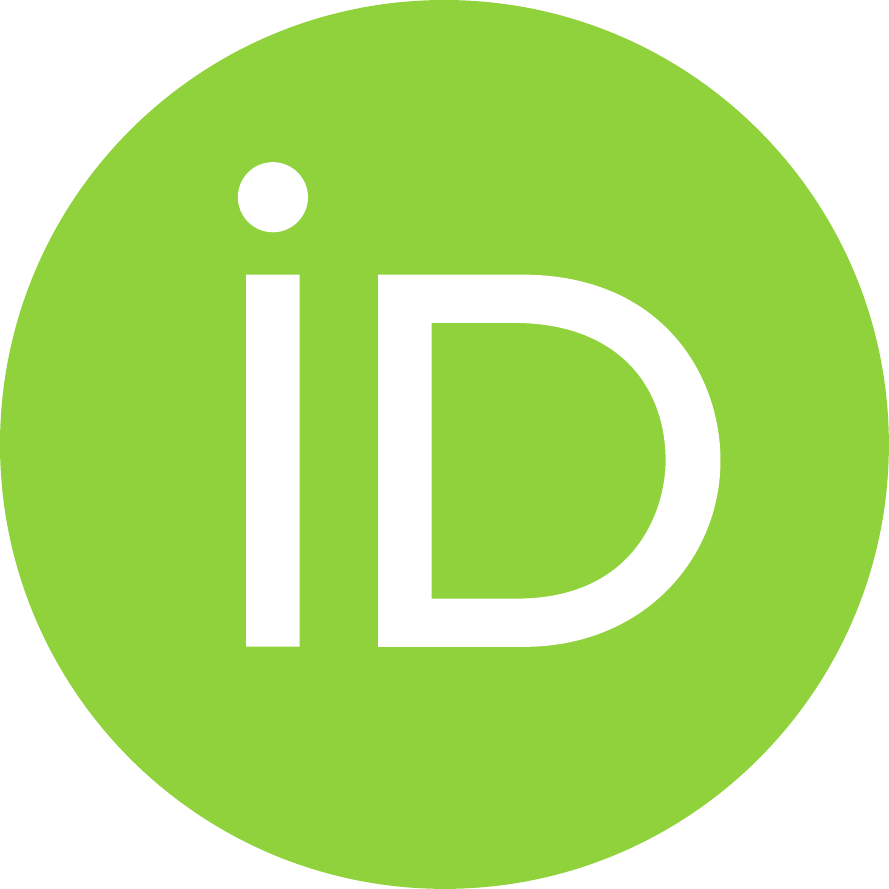}}}

\let\doendproof\endproof
\renewcommand\endproof{~\hfill$\qed$\doendproof}

\begin{document}
\title{Augmenting Geometric Graphs with Matchings} 
%
%
\author{Alexander Pilz\inst{1} \and
Jonathan Rollin\inst{2}\orcidID{0000-0002-6769-7098} \and
Lena Schlipf\inst{3}\orcidID{0000-0001-7043-1867}\thanks{This research is supported by the Ministry of Science, Research and the Arts Baden-W\"urttemberg (Germany).} \and
Andr\'e Schulz\inst{2}\orcidID{0000-0002-2134-4852}}
\authorrunning{A. Pilz et al.}
%
\institute{Graz University of Technology\\
\email{apilz@ist.tugraz.at}	 \and
FernUniversit\"at in Hagen\\
\email{\{jonathan.rollin | andre.schulz\}@fernuni-hagen.de} \and
Universit\"at T\"ubingen\\
\email{lena.schlipf@uni-tuebingen.de}
}
\maketitle              
\begin{abstract}
We study noncrossing geometric graphs and their disjoint compatible geometric matchings.
Given a cycle (a polygon) $P$ we want to draw a set of pairwise disjoint straight-line edges with endpoints on the vertices of $P$ so that these new edges neither cross nor contain any edge of the polygon.
We prove {\sf NP}-completeness of deciding whether there is such a perfect matching.
For any $n$-vertex polygon, with $n\geq 4$, we show that such a matching with $< n/7$ edges is not maximal, that is, it can be extended by another compatible matching edge.
We also construct polygons with maximal compatible matchings with $n/7$ edges, demonstrating the tightness of this bound.
Tight bounds on the size of a minimal maximal compatible matching are also obtained for the families of $d$-regular geometric graphs for each $d\in\{0,1,2\}$.
Finally we consider a related problem. We prove that it is {\sf NP}-complete to decide whether a noncrossing geometric graph $G$ admits a set of compatible noncrossing edges such that $G$ together with these edges has minimum degree five.

\end{abstract}
\section{Introduction}
A geometric graph is a graph drawn in the plane with straight-line edges. Throughout this paper we additionally assume that all geometric graphs are noncrossing. Let $G$ be a given (noncrossing) geometric graph $G$. We want to augment $G$ with a geometric matching on the vertices of $G$ such that no edges cross in the augmentation.
We call such a (geometric) matching \emph{compatible} with~$G$.
Note that our definition of a compatible matching implies that the matching is noncrossing and avoids the edges of~$G$.
Questions regarding compatible matchings were first studied by Rappaport~et~al.~\cite{Rap89,RIT86}.
Rappaport~\cite{Rap89} proved that it is {\sf {\sf NP}}-hard to decide whether for a given geometric graph $G$ there is a compatible matching $M$ such that $G+M$ is a (spanning) cycle. 
Recently Akitaya~et~al.~\cite{AKRST19} confirmed a conjecture of Rappaport and proved that this holds even if~$G$ is a perfect matching.
Note that in this case also $M$ is necessarily a perfect matching.
However, for some compatible perfect matchings $M$ the union $G+M$ might be a collection of several disjoint cycles. 
There are graphs $G$ that do not admit any compatible perfect matching, even when $G$ is a matching.
Such matchings were studied by Aichholzer~et~al.~\cite{CompMatchingConj} who proved that each $m$-edge perfect matching $G$ admits a compatible matching of size at least $\frac{4}{5} m$.
Ishaque~et~al.~\cite{disjoint_matchings} confirmed a conjecture of Aichholzer~et~al.~\cite{CompMatchingConj} which says that any perfect matching $G$ with an even number of edges admits a compatible perfect matching.
For a geometric graph $G$ let $d(G)$ denote the size of a largest compatible matching of $G$ and for a family $\cF$ of geometric graphs  let~$d(\cF)=\min\{d(G)\mid G\in\cF\}$.
Aichholzer~et~al.~\cite{aght-cmgg-11} proved that for the family~$T_n$ of all $n$-vertex geometric trees $\frac{1}{10}n \leq d(T_n) \leq \frac{1}{4} n$ holds and for the family~$P_n$ of all $n$-vertex simple polygons~$\frac{n-3}{4} \leq d(P_n) \leq \frac{1}{3} n$ holds. 

We continue this line of research and consider the following problems. Given a polygon, we first show that it is {\sf NP}-complete to decide whether the polygon admits a compatible perfect matching. %
%
Then we ask for the \enquote{worst} compatible matchings for a given polygon.
That is, we search for small maximal compatible matchings, where a compatible matching $M$ is maximal if there is no compatible matching $M'$ that contains $M$.
We study such matchings also for larger families of $d$-regular geometric graphs.
%


The first studied problem can also be phrased as follows: Given a geometric cycle, can we add edges to obtain a cubic geometric graph? 
In the last section, we consider a related augmentation problem. Given a geometric graph, we show that it is {\sf NP}-complete to decide whether the graph can be augmented to a graph of minimum degree five. The corresponding problem for the maximum vertex degree
asks to add a \emph{maximal} set of edges to the graph such that the maximum vertex degree is bounded from above by a constant. This problem is also known to be {\sf NP}-complete for maximum degree at most~seven~\cite{jansen}.

A survey of Hurtado and T\'oth~\cite{HT13} discusses several other augmentation problems for geometric graphs.
Specifically it is {\sf NP}-hard to decide whether a geometric graph can be augmented to a cubic geometric graph~\cite{p-acg-12} and also whether an abstract planar graph can be augmented to a cubic planar graph (not preserving any fixed embedding)~\cite{HRR15}.
Besides the problems mentioned in that survey, decreasing the diameter~\cite{CGKPSTW17} and the continuous setting (where every point along the edges of an embedded graph is considered as a vertex) received considerable attention~\cite{BBCGL19,DCGSS17}.

\section{Compatible Perfect Matchings in Polygons}

\begin{theorem}\label{thm:perfMatchingHard}
	Given a simple polygon, it is {\sf NP}-complete to decide whether it admits a compatible perfect matching.
\end{theorem}
\begin{proof}
	The problem is obviously in {\sf NP}, as a certificate one can merely provide the added edges.
	{\sf NP}-hardness is shown by a reduction from \textsc{positive planar 1-in-3-SAT}.
	In this problem, shown to be {\sf NP}-hard by Mulzer and Rote~\cite{mulzer_rote}, we are given an instance of 3-SAT with a planar variable--clause incidence graph (i.e., the graph whose vertices are the variables and clauses, which are connected by an edge if and only if the variable occurs in the clause) 
	and no negative literals;
	the instance is considered satisfiable if and only if there is exactly one true variable per clause.
	
	For a given 1-in-3-SAT formula, we take an embedding of its incidence graph and replace its elements by gadgets.
	We first show that finding compatible matchings for a set of disjoint simple polygons is hard and we then show how to connect the individual polygons to obtain a single polygon.
	
	Our construction relies on a gadget that restricts the possible matching edges of vertices. In particular, we 
	introduce a polygonal chain, whose vertices need to be matched to each other in any perfect matching. 
	This is achieved by the \emph{twin-peaks gadget} as shown in \fig{fig:twin_peaks}.
\begin{figure}[tb]
\centering
\includegraphics[width=\textwidth]{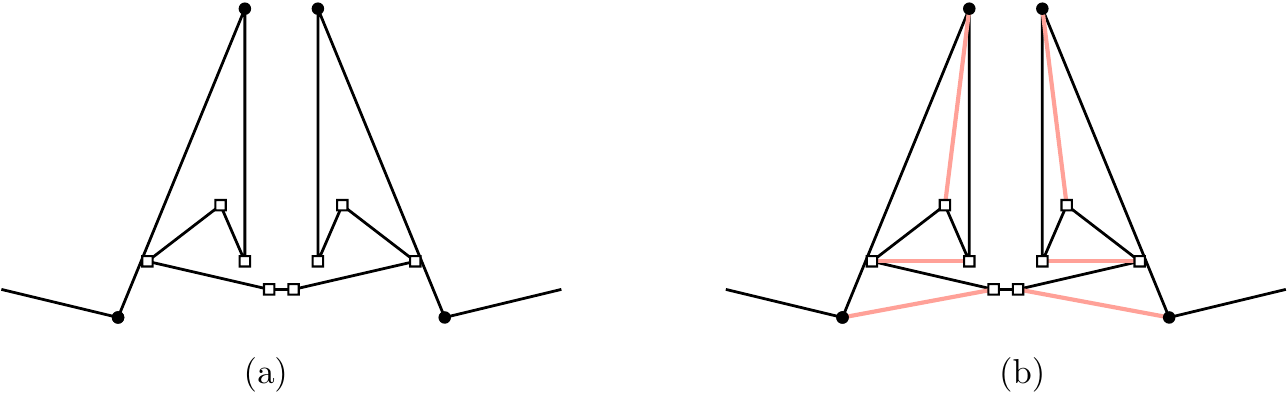}
\caption{(a) This gadget allows for simulating a \enquote{bend} in the polygon without a vertex that needs to be matched.
The construction is scaled such that the eight points marked with squares do not see any other point outside of the gadget (in particular, narrowing it horizontally). (b) A possible matching is shown in red.}
\label{fig:twin_peaks}
\end{figure}
	The gadget is scaled such that the eight vertices in its interior (which are marked with squares in \fig{fig:twin_peaks})  do not see any edges outside of the gadget.
	(We say that a vertex \emph{sees} another vertex if the relative interior of the segment between them does not intersect the polygon.)
	The two topmost vertices must have an edge to the vertices directly below as the vertices below do not see any other (nonadjacent) vertices.
The remaining six \enquote{square} vertices do not have a geometric perfect matching on their own, so any geometric perfect matching containing them must connect them to the two bottommost 
	vertices. Clearly, there is such a matching.
	%
	
	We now present the remaining gadgets (\emph{wire}, \emph{split}, and \emph{clause}) for our reduction.
	The ideas are inspired by the reduction of Pilz~\cite{p-acg-12} who showed that augmenting an arbitrary geometric graph to a crossing-free cubic graph is {\sf NP}-complete. 
	In the following illustrations, vertices of degree two are drawn as a dot.
	Other vertices in the figures represent a sufficiently small twin-peaks gadget.
	
	The \emph{wires} propagate the truth assignment of a variable.
	A wire consists of a sequence of polygons, each containing four vertices of degree two (ignoring twin-peak vertices).
There are only two possible global matchings for these vertices; see \fig{fig:gadgets-wire}(a).
	A \emph{bend} in a wire can be drawn as shown in \fig{fig:gadgets-wire}(b).
	The truth assignment of a wire can be duplicated by a \emph{split gadget};
	see \fig{fig:gadgets-wire}(c).
	A variable is represented by a cyclic wire with split gadgets.
	Recall that in our reduction, we do not need negated variables.
	%
		\begin{figure}[tb]
		\centering
		\includegraphics[width=\textwidth]{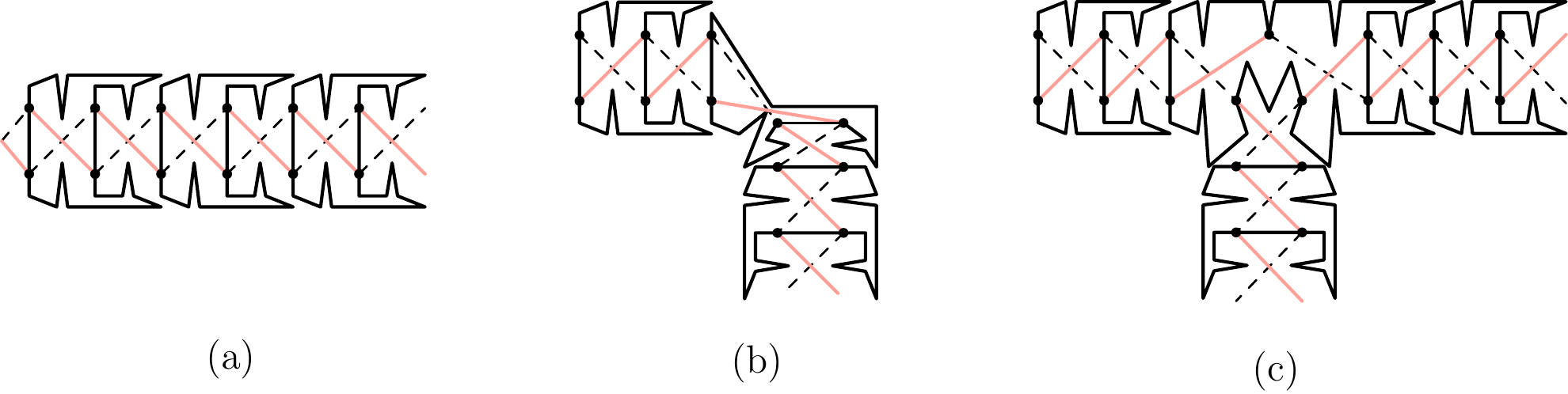}
		\caption{(a) A wire gadget and its two truth states (one in dashed, the other in solid red).
			(b) A bend in a wire gadget. (c) A split gadget that transports the truth setting of one wire to two other ones.
			This is used for representing the variables.}
		\label{fig:gadgets-wire}
	\end{figure}
	The \emph{clause gadget}
	is illustrated in \fig{fig:gadgets}, where the wires enter from the top.
	The vertices there can be matched if and only if one of the vertices is connected to a wire that is in the true state.
	The vertices at the bottom of the gadget make sure that if there are exactly two wires in the false state, then we can add an edge to them.
	Hence, this set of polygons has a compatible perfect matching if and only if the initial formula was satisfiable.

	\begin{figure}[htb]
		\centering
		\includegraphics[width=.8\textwidth,page=2]{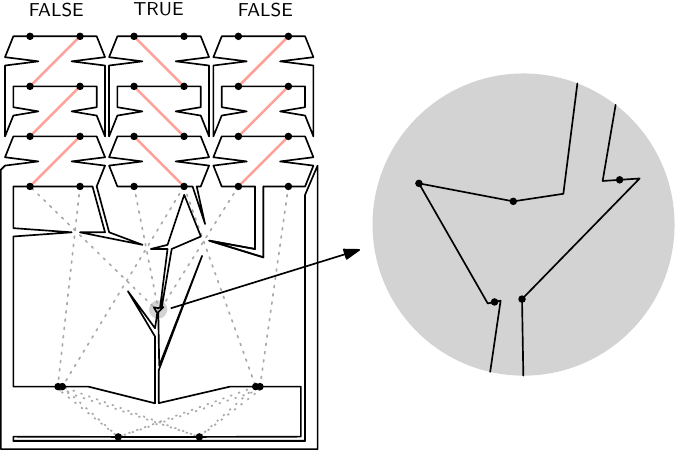}
		\caption{The clause gadget.
			The visibility among the vertices of degree two is indicated by the lighter lines.
			Exactly one vertex of degree two of the part in the circle must be connected to a wire above that carries the true state.}
		\label{fig:gadgets}
	\end{figure}

	It remains to \enquote{merge} the polygons of the construction to one simple polygon.
	Observe that two neighboring polygons can be merged by a small tunnel using four new bends with twin-peaks gadgets line in Fig.~\ref{fig:mergestep}, without affecting the possible compatible perfect matchings of the other vertices.
	We can consider the incidence graph to be connected (otherwise the reduction splits into disjoint problems).
	Hence, we can always merge two distinct neighboring polygons, until there is only a single polygon left.
	\begin{figure}[tb]
		\centering
		\includegraphics{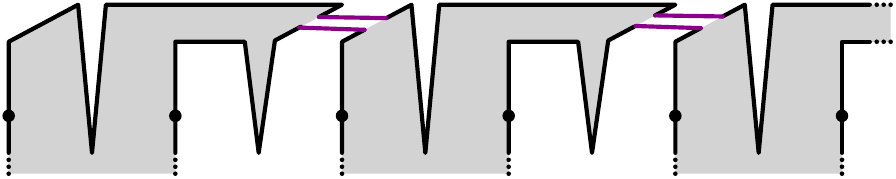}
		\caption{Merging neighboring polygons to a single polygon.}
		\label{fig:mergestep}
	\end{figure}	
\end{proof}

\section{Compatible Maximal Matchings in Geometric Graphs}
 For a geometric graph $G$ let $\mmm(G)$ denote the size of a minimal maximal compatible matching of $G$ and for a family $\cF$ of geometric graphs  let $\mmm(\cF)=\min\{\mmm(G)\mid G\in\cF\}$.
 For a geometric graph $G$ and a maximal compatible matching $M$ we define the following parameters (illustrated in Fig.~\ref{fig:lemma1}):
 \begin{itemize}
 	\item $i_{GM}$ denotes the number of isolated vertices in $G+M$,
 	\item $\Delta_{GM}$ denotes the number of triangular faces in $G$ incident to unmatched vertices only,
 	\item $\sigma_{GM}$ denotes the number of faces of $G+M$ incident to matched vertices only,
 	\item $\nu_{GM}$ denotes the number of edges $uv$ in $G$ where $u$ is unmatched, $v$ is matched, and $uv$ is incident to a reflex angle at $u$ in $G+M$ (see \fig{fig:matchedUnmatchedInc}), 	
 	\item $r^u_{GM}$ and $r^m_{GM}$ denote the number of unmatched and matched vertices incident to a reflex angle in $G+M$, respectively.
 \end{itemize}
 	Here, we call an angle reflex if it is of degree strictly larger than~$\pi$ (there is an angle of degree $2\pi$ at vertices of degree $1$ in $G+M$ and there is no angle considered at isolated vertices).
	Analogically, we call an angle convex if it is of degree $\pi$ or  smaller than~$\pi$.
	
	We assume that the vertices of the considered graphs are in general position. That means that no three vertices are collinear.
\begin{figure}[htb]
	\centering
	\includegraphics{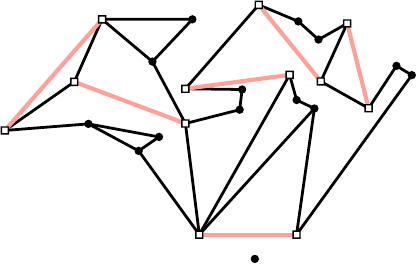}
	\caption{A geometric graph $G$ (black) and a maximal compatible matching $M$ (red). Here $i_{GM}=\Delta_{GM}= 1$, $\sigma_{GM}=2$, $\nu_{GM}=10$, $r^u_{GM}=11$, and $r^m_{GM}=10$.}
	\label{fig:lemma1}
\end{figure}

The following lemma gives a general lower bound on the size of any maximal matching in terms of the parameters introduced above.
We use this bound later to derive specific lower bounds for various classes of geometric graphs below.

\begin{lemma}\label{lem:mmm-generalBound}
 For each geometric graph $G$ and each maximal compatible matching~$M$ of $G$ we have 
  \begin{eqnarray*}2\abs{V(G)} + \nu_{GM}+2\, \sigma_{GM} - r^u_{GM} - 2\,r^m_{GM}  -\; \sum_{\mathclap{u\in V(M)}}\;d_G(u) - \Delta_{GM}  -2 \leq 2 \abs{E(M)}.\end{eqnarray*} 
\end{lemma}

\begin{proof} 
 We subdivide the plane into cells as follows.
 First draw a rectangle enclosing $G$ in the outer face (with four vertices and four edges).
 For each isolated vertex in $G+M$ (one after the other) draw two collinear straight-line edges, both starting at that vertex and until they hit some already drawn edges $e$ and $e'$. 
 The direction of these new edges is arbitrary as long as they do not hit any vertex. Their endpoints become new vertices (subdividing $e$ and $e'$).
 Similarly, for each vertex $u\in V(G)$ incident to some reflex angle in the resulting drawing we draw (one after the other) a straight-line edge starting at $u$.
 The direction of this new edge is chosen such that it cuts the reflex angle at $u$ into two convex angles and such that it stops on some already drawn edge (but not a vertex) which is then subdivided by a new vertex. 
 Avoiding to hit vertices is possible as the points are in general position.
 See \fig{fig:subdivdeMax}.
\begin{figure}[tb]
	\centering
	\includegraphics{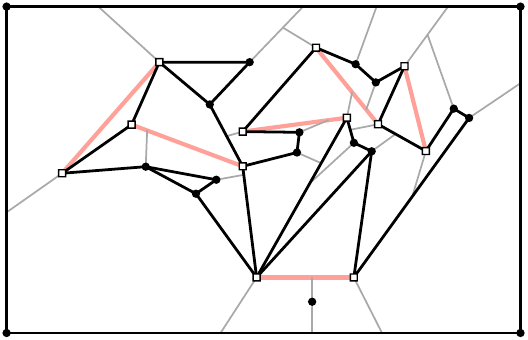}
	\caption{The geometric graph (black) with maximal matching (red) from Fig.~\ref{fig:lemma1} where each reflex angle is cut by a gray edge.}
	\label{fig:subdivdeMax}
\end{figure}
 Let~$D$ denote the final plane graph.
 Then each bounded face in $D$ is convex and $D$ is connected.
 Further, $D$ has exactly  $\abs{V(G)}+ r^u_{GM} + r^m_{GM} + 2\,i_{GM}+4$ vertices and $\abs{E(G)} + \abs{E(M)} + 2 (r^m_{GM} + r^u_{GM} + 2\,i_{GM})+4$ edges (each edge starting at an isolated vertex and each edge cutting a reflex angle creates a new vertex and subdivides an existing edge into two parts).
 By Euler's formula the number $F_D$ of faces in $D$ is exactly
 \begin{equation*}
  F_D = \abs{E(D)}-\abs{V(D)}+2=\abs{E(G)} - \abs{V(G)} +  \abs{E(M)} + r^m_{GM} + r^u_{GM} + 2\,i_{GM} + 2.
 \end{equation*}
 Let $U=V(G)\setminus V(M)$ denote the set of unmatched vertices of $G$ and let $F_i$ denote the number of faces in $D$ with exactly $i$ vertices of $U$ in their boundary.
 Each isolated vertex in $G+M$ is incident to exactly two faces of $D$, each vertex~$u\in U$ not incident to a reflex angle in $G+M$ is incident to exactly $d_G(u)$ faces of $D$, and each remaining vertex $u\in U$ is incident to exactly $d_G(u)+1$ faces of~$D$.
 Therefore
 \begin{equation}
 2\,i_{GM} + r^u_{GM} + \sum_{u\in U}d_G(u) = \sum_{i\geq 1}i\, F_i.\label{eq:1}
 \end{equation}

 Consider two vertices in $U$ incident to a common face $F$ in $D$.
 The line segment connecting these two vertices is an edge of $G$, otherwise $M$ is not maximal.
 So either $F$ has at most two vertices from $U$ or $F$  is a triangular face of $G$ incident to vertices from $U$ only.
 This shows that $F_3=\Delta_{GM}$ and $F_i=0$ for each $i\geq 4$.  
 Further, each face incident to a vertex that is isolated in $G+M$ is not incident to any other unmatched vertex.
 Similarly, for each edge counted by~$\nu_{GM}$ there is a face in $D$ with only one unmatched vertex in its boundary, see \fig{fig:matchedUnmatchedInc}.
 \begin{figure}[tb]
 	\centering
 	\includegraphics[width=.25\textwidth]{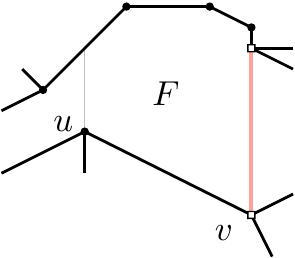}
 	\caption{An edge $uv\in E(G)$ where $u\in V(G)\setminus V(M)$ and $v\in V(M)$ with a reflex angle at~$u$ (in $G+M$). Then $u$ is the only vertex from $V(G)\setminus V(M)$ incident to the face $F$ (obtained by cutting the reflex angle at $u$) since $M$ is maximal.}
 	\label{fig:matchedUnmatchedInc}
 \end{figure}
 Hence~$F_1\geq 2\, i_{GM} + \nu_{GM}$.
 The outer face does not contain any vertices of $U$ and hence~$F_0\geq 1+\sigma_{GM}$.
 Combining these observations with \eqref{eq:1} and $F_2=F_D-F_0-F_1-F_3$ yields
 \begin{align*}
 &2\,i_{GM} + r^u_{GM} + \sum_{u\in U}d_G(u)\\
  = \quad&F_1 + 2\, F_2 + 3\,  \Delta_{GM}\\
  = \quad&2\, F_D-2\, F_0-F_1 +  \Delta_{GM}\\
 \leq\quad &2\abs{E(G)} - 2\abs{V(G)} +  2\abs{E(M)}\\
 & + 2\,i_{GM} + 2\,r^m_{GM} +2\,r^u_{GM} + \Delta_{GM} - \nu_{GM} -2\, \sigma_{GM} +2.
\end{align*}
Now the desired result follows using $\sum\limits_{u\in U}d_G(u)= 2\abs{E(G)} - \sum\limits_{u\in V(M)}d_G(u).$
\end{proof}

The bound of Lemma~\ref{lem:mmm-generalBound} is particularly applicable for regular graphs.

\begin{theorem}\label{thm:mmm-regulargraphs}
	Consider an $n$-vertex geometric graph $G$.
 \begin{enumerate}[label=\bf\alph{enumi}),topsep=2pt]
  \item If $G$ is $0$-regular (a point set) we have $\mmm(G) \geq \frac{n-1}{3}$.
  
  \item If $G$ is $1$-regular (a perfect matching) we have $\mmm(G) \geq \frac{n-2}{6}$.
  
  \item If $G$ is $2$-regular (disjoint polygons) we have $\mmm(G) \geq \frac{n-3}{11}$.
 \end{enumerate}
 All these bounds are tight for infinitely many values of $n$.
\end{theorem}
\begin{proof}
 First consider a $0$-regular $n$-vertex graph $G$ (a point set).
 Then $r^u_{GM}=0$, $r^m_{GM}=2 \abs{E(M)}$, $\nu_{GM}=\Delta_{GM}=0$, and $\sigma_{GM}\geq 0$ for any maximal compatible matching $M$  of~$G$.
 By Lemma~\ref{lem:mmm-generalBound} we have  $2n - 4\abs{E(M)} -2 \leq 2 \abs{E(M)}$.
 This shows $\mmm(G)\geq (n-1)/3$.
 This is tight due to the graphs $G$ and the maximal matchings given in \fig{fig:tightMMM} (left).
 \begin{figure}[tb]
  \centering
  \includegraphics{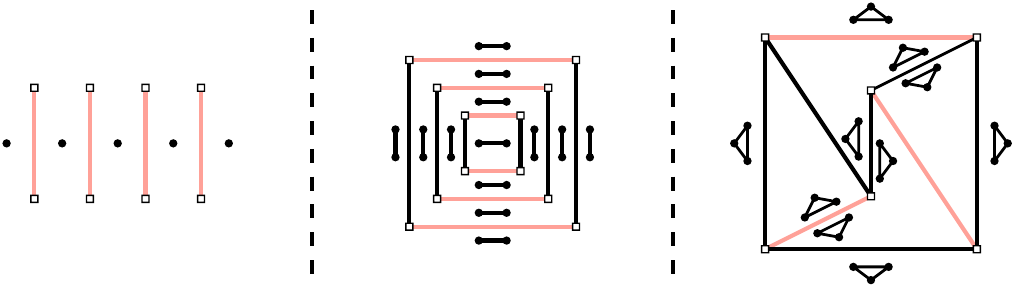}
  \caption{Geometric graphs (black) with minimal maximal compatible matchings (red).}
  \label{fig:tightMMM}
 \end{figure}
 
 Next consider a $1$-regular $n$-vertex graph $G$.
 Each vertex in $G$ is reflex in $G+M$.
 Then $\Delta_{GM}=0$, $\nu_{GM}\geq 0$, $r^u_{GM}=n-2\abs{E(M)}$, $r^m_{GM}=2 \abs{E(M)}$, and $\sigma_{GM}\geq 0$ for any maximal compatible matching $M$ of $G$.
 By Lemma~\ref{lem:mmm-generalBound} we have  $n - 4\abs{E(M)} -2 \leq 2 \abs{E(M)}$.
 This shows $\mmm(G)\geq (n-2)/6$.
 This is tight due to the graphs $G$ and the maximal matchings given in \fig{fig:tightMMM} (middle).
 
 Finally consider a $2$-regular $n$-vertex geometric graph $G$.
 Each vertex in $V(G)\setminus V(M)$ is reflex in $G+M$.
 Then $\nu_{GM}\geq 0$, $r^u_{GM}=n-2\abs{E(M)}$, $r^m_{GM}\leq 2 \abs{E(M)}$, $\sigma_{GM}\geq 0$, and $\Delta_{GM}\leq (n-2\abs{E(M)})/3$ for any maximal compatible matching $M$ of $G$.
 By Lemma~\ref{lem:mmm-generalBound} we have  $n - 6\abs{E(M)} - (n-2\abs{E(M)})/3 -2 \leq 2 \abs{E(M)}$.
 This shows $\mmm(G)\geq (n-3)/11$.
 This is tight due to the graph $G$ and the maximal matching $M$ given in \fig{fig:tightMMM} (right), as an infinite family is obtained by repeatedly replacing an arbitrary triangle with a (scaled) copy of $G+M$.
\end{proof}

\begin{theorem}\label{thm:maxMatchingPoly}
	Let $n\geq 4$ and let $P_n$ denote the family of all $n$-vertex  polygons.
	Then $\mmm(P_n)\geq \frac{1}{7}n$ for all $n$ and this bound is tight for infinitely many values of $n$.
\end{theorem}
\begin{proof}
	The construction in Fig.~\ref{fig:maximalBnd} shows that for infinitely many values of $n$ there is an $n$-vertex polygon with a compatible maximal matching of size $\frac{n}{7}$.
	This shows $\mmm(P_n)\leq \frac{n}{7}$ for infinitely many values of $n$.
	
	\begin{figure}[tb]
		\centering
		\includegraphics{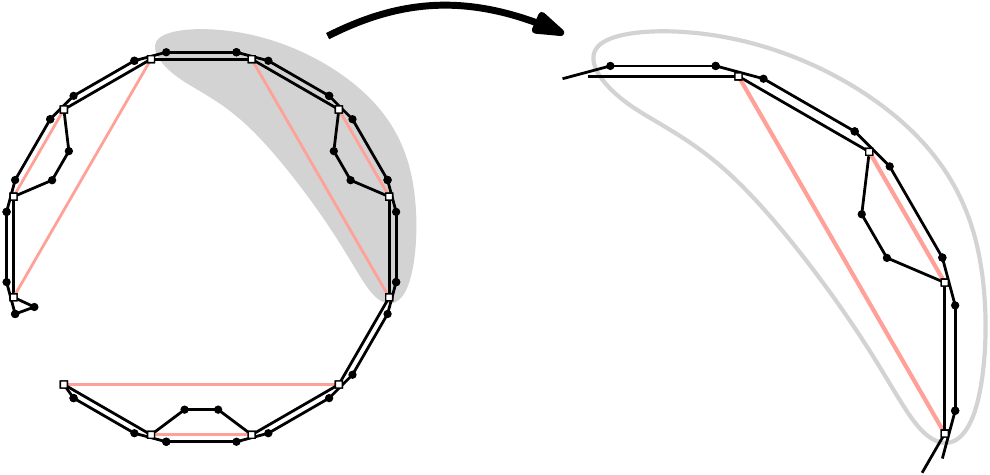}
		\caption{A polygon (black) with a maximal matching (red) with $\frac{n}{7}$ edges (here $n=42$). Note that there are exactly two matching edges between the $14$ vertices in the gray area which can be repeated along a cycle arbitrarily often.}
		\label{fig:maximalBnd}
	\end{figure}

	It remains to prove the lower bound. Let $P$ be an $n$-vertex polygon with a maximal compatible matching~$M$. Since $n\geq 4$,  we have $\abs{E(M)}\geq 1$, $\Delta_{PM}=0$, $r^m_{PM}\leq 2\abs{E(M)}$, and $\sigma_{PM}\geq 0$.
	Let $U=V(P)\setminus V(M)$ denote the unmatched vertices of $P$ and let $E_{UM}$ denote the set of edges $uv$ in $P$ where $u\in U$ and $v\in V(M)$.
	Each vertex in~$U$ has a reflex angle.
	Hence $r^u_{PM}=n-2\abs{E(M)}$ and $\nu_{PM}=\abs{E_{UM}}$. There are $2+\abs{E(M)}$ faces in $P+M$. 
	Each of them either has no vertex from $U$ in its boundary or at least two edges from~$E_{UM}$.
	So~$P+M$ has $2+\abs{E(M)}-\sigma_{PM}$ faces incident to at least two edges from~$E_{UM}$ each.
	Each edge in $E_{UM}$ is on the boundary of two faces of $P+M$.
	Together we have $2\abs{E_{UM}} \geq 2(2+\abs{E(M)}-\sigma_{PM})$ and hence $\nu_{PM}+\sigma_{PM}\geq 2+\abs{E(M)}$.
	Combining these observations with Lemma~\ref{lem:mmm-generalBound} yields~$\abs{E(M)}\geq n/7$, because
	\begin{eqnarray*}
	2n+2+\abs{E(M)}-n-2\abs{E(M)}- 4\abs{E(M)}-4\abs{E(M)}&\leq& \\
	2n+ \nu_{PM}+2\, \sigma_{PM} - r^u_{PM} - 2\,r^m_{PM}  -\; \sum_{\mathclap{u\in V(M)}}\;d_P(u) - \Delta_{PM}  -2 &\leq	& 2 \abs{E(M)}
	\end{eqnarray*}
\end{proof}

For nonregular (abstract) graphs $\hat{G}$ determining a geometric drawing $G$ minimizing $\mmm(G)$ seems harder.
For an integer $n$ and a real number $d$ with $0\leq d\leq 3$, let $\cF_d^n$ denote the family of all (noncrossing) geometric graphs with $n$ vertices and at most $dn$ edges.
Further let $\mmm(d)=\liminf\limits_{n\to\infty}\min\{\mmm(G)/n\mid G\in\cF_d^n\}$.
For each $n$ and each $d\geq 2$ the set $\cF_d^n$ contains a triangulation of a convex polygon (on $2n-3$ edges).
This shows $\mmm(d)=0$ for $d\geq 2$.
Theorem~\ref{thm:mmm-regulargraphs} shows $\mmm(0)=1/3$ and $\mmm(1/2)\leq 1/6$.
The construction  in the following lemma shows $\mmm(d)\leq (2-d)/13$ for $7/10<d<2$.

\begin{lemma}\label{lem:constructionMMM}
	For any integers $m$, $n$ with $n\geq 5$, $\frac{7n+95}{10}\leq m \leq 2n+2$ there is a geometric graph on $n$ vertices and $m$ edges with a maximal compatible matching of size $\ceil{\frac{2n-m+3}{13}}$.	
\end{lemma}
\begin{proof}
	Let $k=\ceil{\frac{2n-m+3}{13}}$.
	Then $k\geq 1$ since $m \leq 2n+2$.
	First suppose that $2n-m+3$ is divisible by $13$, that is,  $m=2n + 3 -13k$.
	We shall construct a geometric graph on $n$ vertices and $m$ edges with a maximal compatible matching of size $k$.
	
	Choose a (noncrossing, geometric) perfect matching $M$ of $2k$ points in convex position and an (inner) triangulation of that geometric graph.
	See \fig{fig:constructionMMM} (left).
	\begin{figure}[tb]
		\centering
		\includegraphics{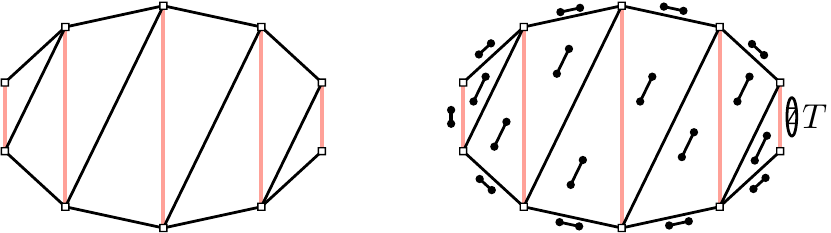}
		\caption{A geometric graph (black) with a maximal compatible matching (red).}
		\label{fig:constructionMMM}
	\end{figure}
	There are $2k-2$ triangular faces and $2k$ edges in the boundary of the outer face.
	Place an isolated edge in the interior of each triangular face.
	Further for all but one of the outer edges $e$ place another (tiny) isolated edge close to $e$ in the outer face (so that there are no visibilities between these).
	So far there are $2k + (4k-4) + (4k-2) =10k-6$ vertices and $(3k-3)+(2k-2)+(2k-1)=7k-6$ edges not in $M$.
	Close to the remaining outer edge we place a triangulation $T$ of a convex polygon on $n-10k+6$ vertices (so that there are no visibilities between these vertices and the isolated edges not in $M$).
	See \fig{fig:constructionMMM} (right).
	Note that $n-10k+6 \geq 2$ since $m\geq \frac{7n+95}{10}$.
	So the graph $T$ contains $2n-20k+9=m-7k+6$ edges.
	The final graph has in total $n$ vertices and $m$ edges not in $M$.
	Further $M$ is a maximal matching by construction.
	
	It remains to consider the case that $2n-m+3$ is not divisible by $13$.
	In this case we apply the construction above with $m'=2n + 3 -13k$ edges.
	To add the remaining $m-m'\leq 12$ edges we replace the triangulation $T$ by an appropriate triangulation of another point set that has some interior points (and hence has more edges).
	\end{proof}

\section{Augmenting to Minimum Degree Five}

In this section, we show that augmenting to a geometric graph with minimum degree five is {\sf NP}-complete.
\begin{theorem}\label{thm:minDeg5Hard}
Given a geometric crossing-free graph $G$, it is {\sf NP}-complete to decide whether there is a set of compatible edges $E$ such that $G+E$ has minimum degree five.
\end{theorem}

\begin{proof}
	The problem is obviously in {\sf NP}, a certificate provides the added edges. {\sf NP}-hardness is shown by a reduction from \textsc{monotone planar rectilinear 3-SAT}. 
		\begin{figure}[htb]
		\centering
\includegraphics{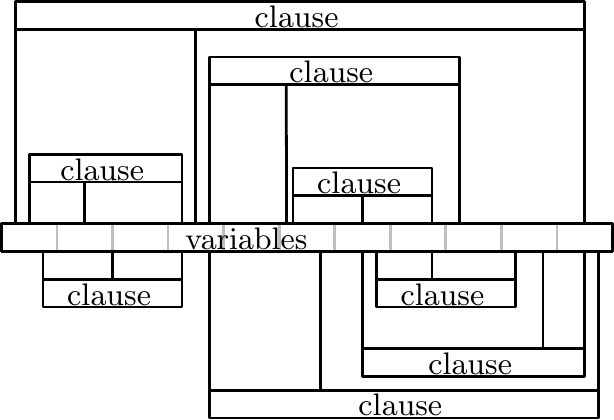}
\caption{A monotone planar 3-SAT instance with a corresponding embedding.}	
\label{fig:mon3sat}
\end{figure}

	In this problem, shown to be {\sf NP}-hard by de Berg and Khosravi~\cite{dBK2010}, we are given an instance of monotone (meaning that each clause has only negative or only positive variables) 3-SAT with a planar variable-clause incidence graph. In this graph, the variable and clause gadgets are
	represented by rectangles. All variable rectangles lie on a horizontal line. The clauses with positive variables lie above the variables and the clauses with negative variables below. The edges
	connecting the clause gadgets to the variable gadgets are vertical line segments and no
	edges cross. 
	See \fig{fig:mon3sat}.
	\begin{figure}[tb]
		\centering	
		\includegraphics[width=0.6\textwidth]{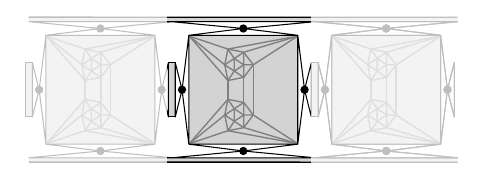}
		\caption{A (geometric) subgraph whose copies will form a wire gadget.}
		\label{fig:wirepiece}
	\end{figure}

For a given monotone planar 3-SAT formula, we take an embedding of its incidence graph (as discussed) and replace its elements by gadgets. 
Note that the corresponding rectilinear layout can be computed in polynomial time and 
has coordinates whose size is bounded by a polynomial~\cite{TT89}.
We use a \emph{wire gadget} that propagates the truth assignments; see \fig{fig:wirepiece}. It consists of a linear sequence of similar subgraphs, each containing exactly four vertices of degree~four (the other vertices have at least degree five). The gray areas contain subgraphs where all vertices have at least degree five. The main idea is that we need to add an edge to each of the 
	vertices of degree four surrounding the big gray squares. But due to blocked visibilities this can only be achieved by a \enquote{windmill} 
	pattern, which has to synchronize with the neighboring parts; see \fig{fig:wirestates}. Thus, we have exactly two ways 
to add edges in order to augment the wire to a graph with minimum degree five. 

	\begin{figure}[tb]
		\centering	
		\includegraphics[width=0.97\textwidth]{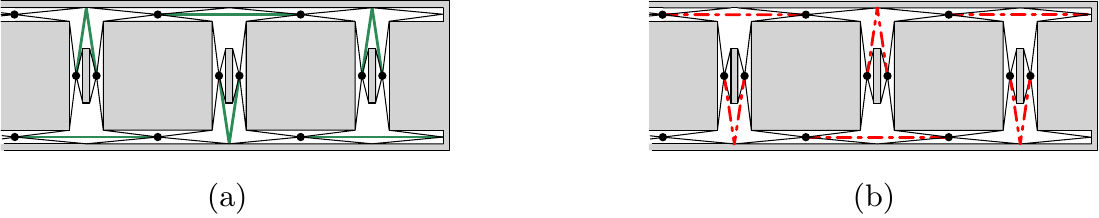}
		\caption{The wire gadget with its only true possible augmentations, associated
		with the assignment true (a) and false (b).}
		\label{fig:wirestates}
	\end{figure}

A bend in a wire is shown in \fig{fig:bendsplit}.
	The truth assignment of a wire can be duplicated by the \emph{split gadget} as shown in \fig{fig:bendsplit}. 
	
	\begin{figure}[tb]
		\centering	
		\includegraphics[width=0.8\textwidth,page=2]{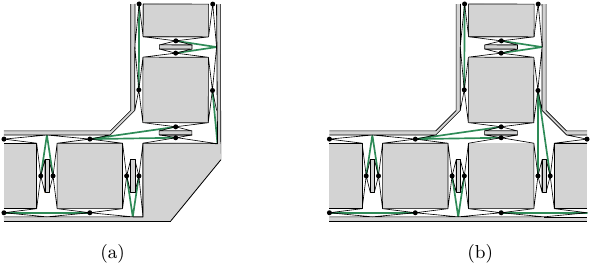}
		\caption{A bended wire (a) and the split gadget (b).}
		\label{fig:bendsplit}
	\end{figure}

A variable is represented by a long wire with split gadgets. 
 Recall that in our reduction, all variables lie on a horizontal line. The clauses with positive variables lie above and the ones with negated variables lie below this line. We can control
 whether a variable or a negated variable is transmitted to the clause gadget by 
 choosing appropriate positions for the corresponding split gadgets. In particular, 
 if we translate the split gadget at the wire by one position to the left or right and keep the truth assignment for the wire,
 the orientation of the augmentation at the position of the new split gadget is flipped.

The \emph{clause gadget} is illustrated in \fig{fig:clause}. The wires enter from left, right and below (respectively above). The 7-gon in the middle of the clause gadget can be augmented to a subgraph with minimum degree five if and only if it is connected to at least one wire in the true state. See also \fig{fig:clausestates}.	
\end{proof}
\begin{figure}[tb]
	\centering
	\includegraphics[width=0.5\textwidth,page=3]{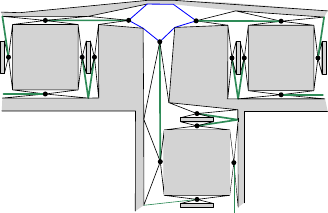}
	\caption{A clause gadget, the three bold segments represent that the corresponding literals are set to true. The central 7-gon (blue) can be augmented to a subgraph of degree at least five if and only if at least one literal is true.}
	\label{fig:clause}
\end{figure}

\begin{figure}[tb]
	\centering
	\includegraphics[width=0.9\textwidth]{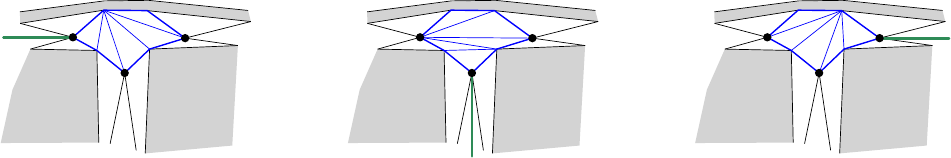}
	\caption{ The three valid possibilities to augment the 7-gon in the clause gadget if one literal is true.}
	\label{fig:clausestates}
\end{figure}

\section{Conclusions}

We study how many noncrossing straight-line matching edges can be drawn on top of a geometric graph $G$ without crossing or using the edges of $G$.
From an algorithmic point of view we show that it is hard to decide whether a perfect matching can be drawn on top of a polygon in this way.
Our results on minimal maximal matchings show that a greedy algorithm will always draw at least $\frac{n}{7}$ edges on top of any $n$-vertex polygon.
However, there are instances where it may draw not more than this amount of edges, although larger compatible matchings exist.

We are interested in how the function $\mmm(G)$ (the size of a minimal maximal compatible matching of $G$) behaves among all geometric graphs $G$ on $n$ vertices and at most $dn$ edges for any value $d\in[0,3]$.
Our results show that degree constraints (like $d$-regularity) help to determine $\mmm(G)$ and also increase the value of $\mmm(G)$ (compared to graphs on the same average degree).
Indeed, we show that any $2$-regular graph has at least $(n-3)/11$ edges in any maximal compatible matching while the construction in Lemma~\ref{lem:constructionMMM} shows that there is a geometric graph $G$ on $n$ vertices with $n$ edges and $\mmm(G)=(n+3)/13$.
We do not know whether there is a family of such geometric graphs with values of~$\mmm(G)$ (asymptotically) even smaller than $n/13$.
It is also not clear for which graphs $\mmm(G)$ is maximized.
For some drawings of empty graphs $G$ we have~$\mmm(G)=\lceil\frac{n}{3}\rceil$. Is this the (asymptotically) largest possible value?


\paragraph*{\bf Acknowledgments.} This work was initiated during the 15th European Research Week on Geometric Graphs (GG Week 2018) in Pritzhagen, Germany.
We thank Kevin Buchin, Michael Hoffmann, Wolfgang Mulzer and Nadja Seiferth for helpful discussions.

 \bibliographystyle{splncs04}
 \bibliography{bib-wg}

\end{document}